\documentclass{amsart}
\usepackage{amssymb,latexsym}
\theoremstyle{plain}
\newtheorem{theorem}{Theorem}

\newtheorem{proposition}{Proposition}
\newtheorem{lemma}{Lemma}
\theoremstyle{definition}

\newtheorem{remark}{Remark}

\newcommand{\enm}[1]{\ensuremath{#1}}          %
\newcommand{\cal}[1]{\mathcal{#1}}

\newcommand{\NN}{\enm{\mathbb{N}}}

\newcommand{\PP}{\enm{\mathbb{P}}}

\newcommand{\Aa}{\enm{\cal{A}}}
\newcommand{\Bb}{\enm{\cal{B}}}

\newcommand{\Ii}{\enm{\cal{I}}}

\newcommand{\Oo}{\enm{\cal{O}}}

\newcommand{\Ss}{\enm{\cal{S}}}

\newcommand{\Uu}{\enm{\cal{U}}}
\newcommand{\Vv}{\enm{\cal{V}}}

\newcommand{\Zz}{\enm{\cal{Z}}}

\renewcommand{\phi}{\varphi}
\renewcommand{\theta}{\vartheta}
\renewcommand{\epsilon}{\varepsilon}

\renewcommand{\to}[1][]{\xrightarrow{\ #1\ }}

\newcommand{\old}[1]{}

\date{}

\begin{document}

\title[$X$-rank]
{Reconstruction of a homogeneous polynomial from its additive decompositions when identifiability fails}
\author{E. Ballico}
\address{Dept. of Mathematics\\
 University of Trento\\
38123 Povo (TN), Italy}
\email{ballico@science.unitn.it}
\thanks{The author was partially supported by MIUR and GNSAGA of INdAM (Italy).}

\begin{abstract}
Let $X\subset \mathbb {P}^r$ be an integral and non-degenerate variety. For any $q\in \PP^r$ let $r_X(q)$ be its $X$-rank
and $\mathcal {S} (X,q)$ the set of all finite subsets of $X$ such that $|S|=r_X(q)$ and $q\in \langle S\rangle$, where
$\langle \ \
\rangle$ denotes the linear span. We consider the case $|\mathcal {S} (X,q)|>1$ (i.e. when $q$ is not $X$-identifiable) and
study the set $W(X)_q:= \cap _{S\in\mathcal {S}}\langle S\rangle$, which we call the non-uniqueness set of $q$. We
study the case $\dim X=1$ and the case $X$ a Veronese embedding of $\mathbb {P}^n$.
\end{abstract}

\maketitle

\section{Introduction}
Let $X\subset \PP^r$ be an integral and non-degenerate variety. For any set $A\subset \PP^r$ let $\langle A\rangle$ denote its
linear span. Fix any $q\in \PP^r$. The \emph{$X$-rank} $r_X(q)$ of $X$ is the minimal cardinality of a finite set $S\subset X$
such that $q\in \langle S\rangle$. The notion of $X$-rank includes the notion of tensor rank of a tensor (take $X$ a
multiprojective space and $X\subset \PP^r$ its Segre embedding) and the notion of additive decomposition of a homogeneous
polynomial or its symmetric tensor rank (take as $X$ a projective space and as $X\subset \PP^r$ one of its Veronese
embeddings). See
\cite{acv, cov0, ik, l}  for a long list of applications of these notions. Let $\Ss (X,q)$ denote the set of all $S\subset X$
such that
$|S|=r_X(q)$ and
$q\in
\langle S\rangle$. Set
$W(X)_q:=
\cap _{S\in
\Ss (X,q)}
\langle S\rangle$. The set $W(X)_q$ is the main actor of this paper. We often write $W_q$ if $X$ is clear from the context.
Note that
$W_q$ is a linear subspace
of $\PP^r$ containing
$q$ and that if $W_q=\{q\}$, and $\Ss (X,q)=\Ss (X,q')$ for some $q'\in \PP^r$, then $q'=q$. If $W_q\ne \{q\}$, then there
are infinitely many $q'\in \PP^r$ such that $\Ss (X,q') =\Ss (X,q)$ (Remark \ref{a1}). We will call $W_q$ the
\emph{non-uniqueness set of $q$}. We have $\dim W_q =r_X(q)-1$ if and only if $\langle S\rangle = \langle S'\rangle$ for
all $S, S'\in \Ss (X,q)$. In particular $W_q =\{q\}$ and $q\notin X$ imply $|\Ss (X,q)|>1$.

We first prove the following two cases (with
$X$ a curve) in which
$W_q =\{q\}$.

\begin{theorem}\label{i1}
Fix an even integer $r\ge 2$. Let $X\subset \PP^r$ be an integral and non-degenerate curve. There is a non-empty open
subset $\Uu \subset \PP^r$ such that $r_X(q) =r/2+1$ for all $q\in \Uu$ and the following properties hold:

\quad (a) We have $\{q\} = \cap _{S\in \Ss (X,q)} \langle S\rangle$ for all $q\in \Uu$.

\quad (b) For all $(q,q')\in \Uu \times \PP^r$ if $\Ss (X,q') =\Ss (X,q)$, then $q'=q$.
\end{theorem}

\begin{theorem}\label{a3}
Fix an integer $d\ge 2$ and let $X\subset \PP^d$ be the rational normal curve. Take any $q\in \PP^d$ such that $\Ss (X,q)$ is
not a singleton.
Then $W_q = \{q\}$. Moreover, if $\Ss (X,q) =\Ss (X,q')$ for some $q'\in \PP^d$, then $q'=q$.
\end{theorem}

Take a non-degenerate $X\subset \PP^r$ and $q\in \PP^r$. For any integer $t>0$ the $t$-secant variety $\sigma _t(X)$ of $X$ is the closure in $\PP^r$ of the union of all linear spaces $\langle S\rangle$ with $S\subset X$ and $|S| =t$. The \emph{border rank} or \emph{border $X$-rank} $b_X(q)$ of $q\in \PP^r$ is the minimal integer $b\ge 1$ such that $q\in
\sigma _b(X)$. We say that a finite set
$A\subset \PP^r$ \emph{irredundantly spans $q$} if $q\in \langle A\rangle$ and $q\notin \langle A'\rangle$ for any
$A'\subsetneq A$. We use Theorem \ref{a3} to prove the following result for the order $d$ Veronese embedding of $\PP^n$.

\begin{theorem}\label{a4.3}
Fix integers $n, d, b, k$, such that $n\ge 2$, $d \ge 8$, $4\le 2b \le d$ and $d+2-b\le k\le 2d-2$. Let $\nu _d:
\PP^n\to
\PP^r$,
$r=\binom{n+d}{n}-1$, be the order
$d$ Veronese embedding. Let $L\subset \PP^n$ be a line. Set $Y:= \nu _d(L)$. Fix $q'\in \langle Y\rangle$ such that
$b_Y(q')=b$ and $r_Y(q') =d+2-b$. Fix a general $U\subset \PP^n$ such that $|U|=k-d-2+b$. Let $q\in \PP^r$ be any point
irredundantly spanned by $\{q'\}\cup \nu _d(U)$. Then:
\begin{enumerate}
\item $r_X(q)=k$, $\Ss (X,q)\supseteq \{E\cup U\}_{E\in \Ss (Y,q')}$ and $\Ss
(X,q)$ uniquely determines $q'$ and $U$.
\item If $k\le 2d-3$, then $\Ss (X,q) =\{E\cup U\}_{E\in \Ss (Y,q')}$ and $W_q = \langle U\cup \{q'\}\rangle$.
\end{enumerate}
\end{theorem}

In the last section we consider the following problem. For any positive integer $t$ let $\Ss (X,q,t)$ be the set of all
$S\subset X$ such that
$|S|=t$ and
$S$ irredundantly spans
$q$. We have $\Ss (X,q,t) = \emptyset$ for all $t< r_X(q)$ and $\Ss (X,q,r_X) = \Ss (X,q)\ne \emptyset$. By the
definition of irredundantly spanning set we have
$\Ss (X,q,t) =\emptyset$ for all $t\ge r+2$. Since $X$ is integral and non-degenerate, for all $(X,q)$ we have $\Ss (X,q,r+1)
\ne
\emptyset$ and $\Ss (X,q,r+1)$ contains a general subset of $X$ with cardinality $r+1$. There are easy examples of
triples $(X,q,t) $ such that
$r> t>r_X(q)$ and $\Ss (X,q,t)=\emptyset$ (Example \ref{q1}). It easy to check that $\Ss (X,q,t)\ne \emptyset$ for all
$t$ such that $r+1-\dim X \le t\le r$ (Lemma \ref{q2}). Set
$W(X)_{q,t}:=
\cap _{S\in
\Ss (X,q,t)}
\langle S\rangle$, with the convention $W(X)_{q,t}:= \PP^r$ if $\Ss (X,q,t) =\emptyset$. We often write $W_{q,t}$ instead of
$W(X)_{q,t}$. 

In section \ref{Sq} we prove the following result.

\begin{proposition}\label{q3}
Let $X\subset \PP^r$, $r\ge 4$, be an integral and non-degenerate curve. Then there exists a non-empty
open subset $\Uu$ of $\PP^r$ such that $W_{q,t} =\{q\}$ for all $q\in \Uu$ and all $\lfloor (r+2)/2\rfloor \le t \le r$.
\end{proposition}

\section{Preliminary observations}

 The
\emph{cactus rank} or \emph{cactus $X$-rank} $c_X(q)$ of $q\in \PP ^r$ is the minimal degree of a zero-dimensional scheme $Z\subset
X$ such that $q\in \langle Z\rangle$. Let $\Zz (X,q)$ denote the set of all zero-dimensional schemes $Z\subset X$ such that
$\deg (Z)=c_X(q)$ and $q\in \langle Z\rangle$.

\begin{remark}\label{a1}
Take any $q\in \PP^r$ and any $S\in \Ss (X,q)$. For any $o\in \langle S\rangle$ the integer $r_X(o)$ is the minimal cardinality
of a subset $S_1\subsetneq S$ such that $o\in \langle S_1\rangle$. We have $S_1\in \Ss (X,o)$.
Now assume $W_q\supsetneq \{q\}$. Fix $S\in \Ss (X,q)$. Let $W'_q$ be the set of all $o\in W_q$ not contained in some
$S'\subsetneq S$. Note that $q\in W'_q$. We saw that $W'_q$ is the complement in $W_q$ of at most $r_X(q)$ hyperplanes of
$W_q$. Thus $W'_q$ is an irreducible algebraic set of dimension $\dim W_q$. We saw that each $o\in W_q$ has rank equal to
the minimal cardinality of a subset of $S$ spanning $q$. Thus $r_X(o) =r_X(q)$ for each
$o\in W'_q$. Since $W'_q \subseteq \cap _{S\in \Ss (X,q)} \langle S\rangle$, we get $\Ss
(X,q)\subseteq
\Ss (X,o)$ for each $o\in W'_q$. Since $q\in \langle S\rangle$,
$q\notin \langle S'\rangle$ for any $S'\subsetneq S$ and $S\in \Ss (X,o)$, we get $\Ss (X,o)\subseteq \Ss (X,q)$. Thus $\Ss (X,q)=\Ss (X,o)$.
\end{remark}

\begin{remark}\label{a2} 
Let $X\subset \PP^d$, $d\ge 2$, be a degree $d$ rational normal curve. We use
\cite[\S 1.3]{ik} and \cite{cs} for the following observations. Fix
$q\in \PP^r$.

\quad (i)  We have $b_X(q) =c_X(q)$ (\cite[Lemma 1.38]{ik})  and $|\Zz (X,q)|=1$ (\cite[Part (i) of
Theorem 1.43]{ik}). 

\quad (ii) If $c_X(q)<r_X(q)$, then $c_X(q)+r_X(q) =d+2 $ and $\Ss (X,q)$ is infinite. Let $Z$ be the only
element of $\Zz (X,q)$, $d+2-c_X(q)$ is the minimal degree of a scheme $A\subset X$, such that $q\in \langle A\rangle$ and
$A\nsupseteq Z$.

\quad (iii) If $r_X(q) >c_X(q)$, then $\{q\} =  \langle Z\rangle\cap \langle S\rangle$, where $\{Z\} = \Zz (X,q)$ and $S$ is
any element of
$\Ss (X,q)$ (this also follows from the fact that $h^1(\PP^1,L) =0$ for any line bundle $L$ on $\PP^1$ with $\deg (L)\ge -1$,
as in the proof of Claim 1 below). 

\quad (iv) If $r_X(q) > b_X(q)$, then $\dim \Ss (X,q) =d+3 -2b$ (\cite[eq. (9)]{cs}).

\quad (v) If $r$ is odd and $r_X(q) = (d+2)/2 $ (i.e. $r_X(q)=b_X(q)$ is the generic rank), then $\Ss (X,q)=\Zz (X,q)$ and
$|\Ss (X,q)|=1$ (\cite[Theorem 1.43]{ik}).

\quad (vi) Assume $d$ even and $r_X(q) = d/2+1$ and so $q$ has the generic rank and $b_X(q) =r_X(q)$, but we do not assume that
$q$ is general in $\PP^d$. Fix $S, S'\in \Ss (X,q)$ such that $S\ne S'$. 

\quad {\bf {Claim 1:}} $\langle S\rangle \cap \langle S'\rangle =\{q\}$.

\quad {\bf {Proof of Claim 1:}} Since $S\ne S'$ and $S\in \Ss (X,q)$, we have
$q\notin \langle S\cap S'\rangle$. The Grassmann's formula gives $h^1(\PP^1,\Ii _{S\cup S'}(d)) >0$. Since $h^1(\PP^1,L) =0$
for any line bundle $L$ on $\PP^1$ with $\deg (L)\ge -1$ and $q\notin X$, we have $S\cap S'= \emptyset$ and $h^1(\PP^1,\Ii
_{S\cup S'}(d))=1$. The Grassmann's
formula then implies
$\dim (\langle S\rangle \cap \langle S'\rangle) =0$, proving Claim 1.

Obviously Claim 1 implies $W_q = \{q\}$ in this case, which by \cite[Part (i) of Theorem 1.43]{ik} is the only case in which $r_X(q)=b_X(q)$
and
$\Zz (X,q)$ is not a singleton.

Note that (iii) implies that each $q\in \PP^r$ with $c_X(q)\ne r_X(q)$ is uniquely determined by the zero-dimensional scheme
evincing its cactus rang and by one single set evincing its rank (any $S\in \Ss (X,q)$ would do the job). Obviously part (i)
implies that most $q\in \PP^r$ (the ones with $r_X(q)=b_X(q)$) are not uniquely determined by $\Ss (X,q)$. By Remark \ref{a1}
and part (i) for each $q\in \PP^r$ such that $r_X(q)=b_X(q)$ there are exactly $\infty ^t$, $t:= r_X(q)-1$, points $o\in \PP^r$ with
$\Ss (X,o) =\Ss (X,q)$. A similar remark holds for an arbitrary variety $X$ for the points $q$ such that $|\Ss (X,q)|=1$.
\end{remark}

In the proof of Theorem \ref{a4.3} we use the following result (\cite[Theorem 1]{b}, \cite[Theorem 2]{b1}); we
use the assumption $d\ge 6$ to have $4d-5 \ge 3d+1$ and hence to apply  a small part of \cite[Theorem 1]{b}).

\begin{lemma}\label{a4.5}
(\cite[Theorem 1]{b}, \cite[Theorem 2]{b1}) Fix an integer $d\ge 6$. Let $S\subset \PP^n$, $n\ge 2$, be a finite set such that
$|S| \le 4d-5$. We have $h^1(\Ii _S(d)) >0$ if and only if there is $F\subseteq S$ in one of the following cases:
\begin{enumerate}
\item $|F| = d+1$ and $F$ is contained in a line;
\item $|F|=2d+2$ and $F$ is contained in a reduced conic $D$; if $D=L_1\cup L_2$ with each $L_i$ a line we have $L_1\cap
L_2\notin F$ and $|F\cap L_1|=|F\cap L_2| =d+1$;
\item $|F| = 3d$, $F$ is contained in the smooth part of a reduced plane cubic $C$ and $F$ is the complete intersection of $C$
and a degree $d$ hypersurface;
\item $|F|=3d+1$ and $F$ is contained in a plane cubic.
\end{enumerate}
\end{lemma}

\section{Proofs of Theorems \ref{i1}, \ref{a3} and \ref{a4.3}}

\begin{proof}[Proof of Theorem \ref{i1}:]
To prove part (b) it is sufficient to prove part (a), because $\Ss (X,q') =\Ss (X,q)$ implies $\{q'\} \subseteq W_{q'} =W_q$ and $W_q = \{q\}$ for $q\in \Uu$. 

Since part (a) is trivial in the case
$r=2$, we assume
$r\ge 4$. Since no non-degenerate curve is defective (\cite[Corollary 1.5 and Remark 1.6]{a}), there is a non-empty open subset
$\Vv
\subset \PP^r$ such that $r_X(q)=r/2+1$ and $\dim \Ss (X,q)=1$ for all $q\in \Vv$.

For each set $S\subset X$ such that $|S|=r/2+1$ and $\dim \langle S\rangle =r/2$ let $\ell _S: \PP^r\setminus \langle S\rangle
\to \PP^{r/2-1}$ denote the linear projection from $\langle S\rangle$. For a general $S$ we have $\langle S\rangle \cap X =S$
(scheme-theoreticaly) by Bertini's theorem and the trisecant lemma  (\cite[Corollary 2.2]{r}) and $\ell _{S|X\setminus S}$ is
birational onto its image, again by the trisecant lemma and the assumption $r\ge 4$.  Let $X_S\subset \PP^{r/2-1}$ be the
closure of $\ell _S(X\setminus S)$ in $\PP^{r/2-1}$. There is a finite set $E\subset X_S$ containing $X_S\setminus \ell
_S(X\setminus S)$ and such that for each $p\in X_S\setminus E$ there is a unique $o\in X\setminus S$ such that $\ell _S(o)
=p$. For any set
$A\subset X_S\setminus E$ let $A_S \subset X\setminus S$ denote the only set such that $\ell _S(A_S) =A$.
Any general $A\subset X_S\setminus E$ such that $|A|=r/2+1$ is linearly dependent, but each proper subset of $A$ is linearly
independent. Thus $\langle S\rangle \cap \langle A_S\rangle$ is a single point, $q_{S,A}$, and $q_{S,A}\notin \langle B\rangle$
for any $B\subsetneq A_S$. For a general $A$ we get as $A_S$ a general subset of $X$ with cardinality $r/2+1$. Thus for a
general $A$ we have
$q_{S,A}\notin S'$ for any $S'\subsetneq S$. We start with $S\in \Ss (X,o)$ for a general $o\in \PP^r$. Thus $r_X(q)=r/2+1$
for a general $q\in \langle S\rangle$. Thus for a general $A$ we get $S\in \Ss (q_{S,A})$ and $A_S\in \Ss (q_{S,A})$. By construction we have
$\{q_{S,A}\} = \langle S\rangle \cap \langle A_S\rangle$. For a general $A$ the point $q_{S,A}$ is general in $\langle
S\rangle$. By the generality of $S$ we get that the points $q_{S,A}$'s (with $(S,A)$ varying, but general), cover a non-empty
open subset of
$\PP^r$.
\end{proof}

\begin{proof}[Proof of Theorem \ref{a3}:]
Set $b:= b_X(q)$. Since $\Ss (X,q)$ is
not a singleton, we have $q\notin X$ and hence $b \ge 2$. Part (vi) of Remark \ref{a2} covers the case $r_X(q)=b$ and
hence we may assume $r_X(q)>b$. Thus $r_X(q)=d+2-b$. By part (v) of Remark \ref{a2} we have $\dim \Ss (X,q)\ge 2$.
We will prove the stronger assumption that $\{q\} = \cap _{A\in \Gamma} \langle A\rangle$, where $\Gamma$ is any irreducible family
contained in $\Ss (X,q)$ and with $\dim \Gamma = d+3-2b$; we do not assume that $\Gamma$ is closed in $\Ss (X,q)$.

\quad (a) First assume $b=2$. We use the proof of \cite[Proposition 5.1]{lt}. Fix $a\in \PP^d\setminus \{q\}$.
Let $H\subset \PP^d$ be a general hyperplane containing $q$. Since $q\notin X$, Bertini's and Bezout's theorems give that
$X\cap H$ is formed by $d$ distinct points. Since $X$ is connected, the exact sequence $$0 \to \Ii _X\to \Ii _X(1)\to \Ii
_{X\cap H,H}(1)\to 0$$gives that $X\cap H$ spans $H$. Thus $q\in \langle X\cap H\rangle$. Since $r_X(q)=d$, we get $X\cap H\in
\Ss (X,q)$. The generality of $H$ gives $a\notin H$, concluding the proof that $W_q = \{q\}$.

\quad (b) Step (a) and part (i) (resp. part (vi)) of Remark \ref{a2} for the case $d$ odd (resp. $d$ even) and $q$ with generic
rank cover all cases with $d\le 4$. Thus we may assume $d\ge 5$ and use induction on $d$. Fix a general $o\in X$. Let
$\ell _o: \PP^d\setminus \{o\}\to \PP^{d-1}$ denote the linear projection from $o$. Let $Y\subset \PP^{d-1}$ denote the
closure of $\ell _o(X\setminus \{o\})$ in $\PP^{d-1}$. $Y$ is a rational normal curve of $\PP^{d-1}$. Set $q':= \ell _o(q)$
and $Z':= \ell _o(Z)$  (by the generality of $o$ we have $o\notin \langle Z\rangle$ and hence $Z'$ is well-defined, $\deg
(Z') =b$ and $\dim \langle Z'\rangle =b-1$). The generality of $o$ also implies that $q\notin \langle Z''\cup \{o\}\rangle$ for
any $Z''\subsetneq Z$ (here we use that $X$ is a smooth curve and hence $Z$ has only finitely many subschemes). Thus $q'\in
\langle Z'\rangle$ and
$q'\notin
\langle Z''\rangle$ for all $Z''\subsetneq Z'$. Since $Y$ is a degree $d-1$ rational normal curve and $b \le d/2$, parts (i) and (ii) of
Remark \ref{a2} imply $b_Y(q') =b$ and $\Zz (Y,q') =\{Z'\}$.
Fix an irreducible family $\Gamma \subseteq \Ss (X,q)$ such that $\dim \Gamma =d+3-2b$ (it exists by part (v) of Remark \ref{a2}) . Let $\Bb$ denote the set of all
$A\in
\Gamma$ such that
$o\in A$. By part (v) of Remark
\ref{a2} and the generality of
$o$ we have $\Bb \ne \emptyset$ and $\dim \Bb = d+2-2b$. Set $\Aa := \{\ell _o(B\setminus \{o\})\}_{B\in \Bb}$. Since $Y$ is a rational normal curve, parts
(i) and (ii) of Remark \ref{a2} imply $\Aa \subseteq \Ss (Y,q')$. We have $\dim \Aa = (d-1)+3-2b$. The inductive assumption gives $\{q'\} = \cap _{A\in \Aa} \langle A\rangle$.
Thus $\langle \{o,q\}\rangle =\cap _{B\in \Bb} \langle B\rangle$. Since $\dim \Gamma =\dim \Ss (X,q)=d+3-2b$ (part (v) of Remark \ref{a2}) and $o$ is general in $X$, there is $S\in \Gamma$ such that $o\notin S$. Thus $\cap_{A\in \Gamma} \langle A\rangle =\{q\}$.
\end{proof}

\begin{proof}[Proof of Theorem \ref{a4.3}:]

By Autarky (\cite[Exercise 3.2.2.2]{l}) we may assume $U\ne \emptyset$. Since $U$ is general in $\PP^n$, we have $\dim \langle U\cup Y\rangle = \min
\{r,\dim
\langle Y\rangle +|U|\}$. Since
$\dim
\langle Y\rangle =d$ and $d+|U| < r$, we have $\langle U\rangle \cap \langle Y\rangle =\emptyset$. By Theorem
\ref{a3} we have
$W(Y)_{q'} =
\{q'\}$. Take
$E\in
\Ss (Y,q')$ and set $A:= U\cup E$.  The set $\{q'\}\cup U$ irredundantly spans $q$ and $\langle Y\rangle \cap \langle
U\rangle =
\emptyset$, we have $E\cap U=\emptyset$ and hence $|A| = k$. 
Since $|A|=k$ and $q\in \langle \nu _d(A)\rangle$, we have $r_X(q)\le k$. Since $U$ is general in $\PP^n$, we have $h^0(\Ii _A(t)) = \max \{0,h^0(\Ii _E(t)) -|U|\}$ for all $t\in \NN$; to use this equality we need to fix one element, $E$, of $\Ss (Y,q')$, before choosing a general $U$.

\quad (a) In this step we prove that $r_X(q)=k$ and that $\Ss (X,q) =\{E\cup U\}_{E\in \Ss (Y,q')}$ if $k\le 2d-3$. Note that
we have $W_q = \langle \nu _d(U)\cup \{q'\}\rangle$ for any $q$ such that $\Ss (X,q) =\{E\cup U\}_{E\in \Ss (Y,q')}$ by
Theorem \ref{a3}. Assume either $r_X(q)<k$ or $k\le 2d-3$ and the existence of $B\in \Ss (X,q) \setminus\{E\cup U\}_{E\in \Ss
(Y,q')}$. In the former case take $B\in \Ss (X,q)$. Set
$S:= A\cup B$. In both cases we have $|B|\le |A|$ and $|A|+|B| \le 4k-5$. Since
$h^1(\Ii _S(d)) >0$ (\cite[Lemma 1]{bb}) there is $F\subseteq S$ in one of the cases listed in Lemma \ref{a4.5}. 

\quad (a1) Assume the existence of a plane cubic $T\subset \PP^n$ such that $|T\cap S|\ge 3d$. 

\quad (a1.1) Assume $n=2$. Thus $T$ is an effective divisor of $\PP^n$. Consider the residual exact sequence of $T$ in
$\PP^2$:
\begin{equation}\label{eqb+1}
0 \to \Ii _{S\setminus S\cap T}(d-3)\to \Ii _S(d)\to \Ii _{S\cap T,T}(d)\to 0
\end{equation}
Since $|S\setminus S\cap T| \le 4d-5-3d = d-5$, we have $h^1(\Ii _{S\setminus S\cap T}(d-3))=0$. Thus either \cite[Lemma
5.1]{bb2}
or \cite[Lemmas 2.4 and 2.5]{bbcg1} give $A\setminus A\cap T =B\setminus B\cap T$. Assume for the moment $L\nsubseteq T$.
Bezout
gives $|L\cap T|\le 3$. Since $U$ is general and $h^0(\Oo _{\PP^2}(3)) =10$, we have $|U\cap T| \le 9$. Thus $|B\cap T|\ge
3d -12 > 12\ge |D\cap A|$ and hence $|B|>|A|$, a contradiction. Now assume $L\subset T$. Since $h^0(\Oo _{\PP^2}(2)) = 6$
and $U\cap L=\emptyset$, we get $|A\cap T|\le d+8-b$. Thus $|B\cap T|\ge 2d-5+b$ and again $|B|>|A|$, a contradiction.

\quad (a1.2) Assume $n>2$. Let $M\subset \PP^n$ be a general hyperplane containing the plane $\langle T\rangle$ (so $M
= \langle T\rangle$ if $n=3$). Since $S$ is a finite set and $M$ is a general hyperplane containing $\langle T\rangle$, we have
$S\cap M = S\cap \langle T\rangle$.  Consider the residual exact sequence of $M$ in $\PP^n$:
\begin{equation}\label{eqb+2}
0 \to \Ii _{S\setminus S\cap M}(d-1) \to \Ii _S(d)\to \Ii _{S\cap M,M}(d)\to 0
\end{equation}
Since $|S\setminus A\cap M| \le 4d-5-3d = d-5$, we have $h^1(\Ii _{S\setminus S\cap M}(d-1))=0$.  Thus either \cite[Lemma
5.1]{bb2}
or \cite[Lemmas 2.4 and 2.5]{bbcg1} give $A\setminus A\cap M =B\setminus B\cap M$. Since no $4$ points of $U$ are coplanar,
we have $|A\cap M| \le d+5-b < 3d-d-2 +b$. Thus $|B|>|A|$, a contradiction.

\quad (a2) Assume the existence of a plane conic $D$ such that $|S\cap D|\ge 2d+2$. 

\quad (a2.1) Assume
$n=2$. Consider the residual exact sequence of $D$ in $\PP^2$:

\begin{equation}\label{eqb+3}
0 \to \Ii _{S\setminus S\cap D}(d-2) \to \Ii _S(d)\to \Ii _{S\cap D,D}(d)\to 0
\end{equation}

First assume $h^1(\Ii _{S\setminus S\cap D}(d-2))>0$. Since $|S\setminus S\cap D|\le 4d-5 -2d-2 =2(d-3)-1$,
there is a line $R\subset \PP^2$ such that $|R\cap (S\setminus S\cap D)|\ge d-1$ (\cite[Lemma 34]{bgi}). Thus $|S\cap (D\cup
R)|\ge 3d+1$. Step (a1) gives a contradiction. Now assume $h^1(\Ii _{S\setminus S\cap D}(d-2))=0$. Either \cite[Lemma
5.1]{bb2}
or \cite[Lemmas 2.4 and 2.5]{bbcg1} give $A\setminus A\cap D =B\setminus B\cap D$. Assume for the moment $L\nsubseteq D$.
Thus $|L\cap D|\le 2$. Since $U$ is general and $h^0(\Oo _{\PP^2}(2)) =6$, we have $|U\cap D| \le 5$. Thus $|B\cap D| > |A\cap
D|$ and so $|B|>|A|$, a contradiction. Now assume $L\subset D$. Write $D =L\cup R$ with $R$ a line. Since $U\cap L=\emptyset$,
$|L\cap D| =1$ and
$|U\cap R'|\le 2$ for each line
$R'$, we have $|A\cap D| \le d+4-b$ and $|A\cap L|\ge b-2$. If $b\ge 4$ (resp. $b\ge 3$) we get $|A|<|B|$ (resp. $|A|\le |B|$); we also conclude if $k\le 2d-4$ or $|A\cap B| \ge 4$.  
Thus we may exclude these cases; we only need to assume that $S\nsubseteq D$ (i.e. $A\nsubseteq D$, i.e. $B\nsubseteq D$).
From now on in this step (a2.1) we assume
$2\le b\le 4$ and
$A\nsubseteq D$. By Remark \ref{a2} there is a unique zero-dimensional scheme $Z_1\subset L$ such that $\deg (Z_1) =b$ and
$q'\in
\langle
\nu _d(Z_1)$. Set
$Z_2:= Z_1\cup E$ and
$Z:= Z_2\cup Z$. Since $Z_1$ is unique and we take $U$ general after fixing $q'$, no line $F\subset \PP^n$ spanned by $2$
points of $U$ meets $Z_1$. Hence $\deg (Z_1\cap F) \le 2$ for each line $F\ne L$. Note that $\deg (Z)\le 2k-d-2+2b \le
3d-6+2b$. By
\cite[Lemma 1]{bb} we have
$h^1(\Ii _Z(d)) >0$.
Consider the residual exact sequence of $R$ in $\PP^2$:

\begin{equation}\label{eqb+3.1}
0 \to \Ii _{\mathrm{Res}_R(Z)}(d-1) \to \Ii _Z(d)\to \Ii _{Z\cap R,R}(d)\to 0
\end{equation}
Since $B$ is a finite set, obviously $\mathrm{Res}_R(Z) = \mathrm{Res}_R(Z_2)\cup \mathrm{Res}_R(B)$ and $\mathrm{Res}_R(B)$.
We have $\mathrm{Res}_R(Z_2) \subseteq Z_1\cup (U\setminus U\cap L)$. Since $U\cap L=\emptyset$, we have $U\setminus U\cap R
=U\setminus U\cap D = A\setminus A\cap D =B\setminus B\cap D$. Since $|S\cap D| \ge 3$, $|E\cap R|\le 1$ and $|U\cap R|\le 2$,
we have $|B\cap R|\ge d-1 > 1+|A\cap R|$. Since $\deg (Z)\le 3d-6+2b$, $\deg (S\cap R)\ge d+2$ and $\deg (E\cap R)\le 1$,
we have
$\deg (\mathrm{Res}_R(Z)) \le 2d-7+2b$. Thus the inequality
$|B|\le |A|$ gives
$|B\cap L|
\le d-b$. Hence
$\deg (Z\cap L)
\le d$. The generality of $U$ implies that no line intersects $\mathrm{Res}_R(Z)$ in a scheme of degree $\ge d+1$ and no conic
intersects $\mathrm{Res}_R(Z)$ in a scheme of degree $\ge 2d-2$. By \cite[Corollaire 2]{ep} we have $h^1(\Ii
_{\mathrm{Res}_R(Z)}(d-1)) =0$. By \cite[Lemma 5.1]{bb2} or \cite[Lemma 2.5]{bbcg1} we have $\mathrm{Res}_R(Z_2) = B\setminus
B\cap R$. Since $Z_1$ is not reduced, we get
that $Z_1$ has a unique non-reduced connected component, that this connected component has degree $2$ and that $R$ meets
it. The generality of $U$ gives $|U\cap R| \le 1$. Since
$\mathrm{Res}_D(Z_2)=\mathrm{Res}_D(A) = U\setminus U\cap D$, we get $|B|>|A|$, a contradiction.

\quad (a2.2) Assume $n>2$. Let $M\subset \PP^n$ be a general hyperplane containing the plane $\langle D\rangle$. Thus $S\cap M
= S\cap \langle D\rangle$. Since $U$ is
general, no $4$ points of $U$ are coplanar. Thus $|U\cap M| = |U\cap \langle D\rangle| \le 3$.

\quad (a2.2.1) Assume $h^1(\Ii _{S\setminus S\cap M}(d-1)) >0$. Since  $|S\setminus S\cap M|\le |A|+|B| -2d-2 \le 2(d-1)+1$, there is a line $R'\subset \PP^n$ such that
$|R'\cap (S\setminus S\cap M)|\ge d+1$. If $R'\subset \langle D\rangle$, then $R'\cup D$ is a plane cubic and we may apply step (a1). Thus we may assume $R'\nsubseteq \langle D\rangle$. Let $N\subset \PP^n$ be a general hyperplane containing $N$. Since $S$ is a finite set, the generality of $M$ and $N$ gives $S\cap (M\cup N) =S\cap (\langle D\rangle \cup R')$. Consider the residual exact sequence
\begin{equation}\label{eqb+4}
0 \to \Ii _{S\setminus S\cap (M\cup N)}(d-2) \to \Ii _S(d)\to \Ii _{S\cap (M\cup N),M\cup N}(d)\to 0
\end{equation}
of $M\cup N$ in $\PP^n$. Since $|S\setminus S\cap (M\cup N)|\le |A|+|B| -3d-3 \le d-1$, we have $h^1(\Ii _{S\setminus S\cap (M\cup N)}(d-2))=0$. Thus either \cite[Lemma
5.1]{bb2}
or \cite[Lemmas 2.4 and 2.5]{bbcg1} give $A\setminus A\cap (M\cup N) =B\setminus B\cap (M\cup N)$. We have $A\cap (M\cup N) \subseteq E\cup
(U\cap( M\cup N))$ and hence $|A\setminus A\cap (M\cup N)|\ge k-d-2+b-5$. Since $|A\cap (M\cup N)| \le d+7-b$, we get $|B\cap
(M\cup N)| \ge 3d+3 -d-7+b = 2d-4+b$. Since $|B\setminus B\cap (M\cup N)|\ge k-d-2+b-5$, we get a contradiction.

\quad (a2.2.2)  Assume $h^1(\Ii _{S\setminus S\cap M}(d-1)) =0$. Either \cite[Lemma
5.1]{bb2}
or \cite[Lemmas 2.4 and 2.5]{bbcg1} give $A\setminus A\cap M =B\setminus B\cap M$. Since $|U\cap M| = |U\cap \langle D\rangle| \le 3$, we have
$|U\setminus U\cap M|\ge k-d-5+b$. Assume for the moment $L\nsubseteq \langle D\rangle$. We get $|E\cap M| \le 1$ and hence $|A\setminus A\cap M| \ge k-4$.  
Since $A\setminus A\cap M =B\setminus B\cap M$, we get $|S\cap M| \le |A|+|B| -2k-8$ and hence $2d+2\le 8$, a contradiction.
 Now assume $L\subset \langle D\rangle$. If $L\nsubseteq D$ we get (since $|L\cap D|\le 2$) $|S\cap M| \ge 3d+b$. Since $A\setminus A\cap M =B\setminus B\cap M$
 and $|U\setminus U\cap M| \ge k-d-1+b$, we get $|S| \ge 2d+2b+k-1$, a contradiction.

\quad (a3) Assume the existence of a line $R\subset \PP^n$ such that $|R\cap S| \ge d+2$. Since $U$ is a general subset of $\PP^n$ with cardinality
$k-d-2+b$, no $3$ of its points are collinear and $U\cap L =\emptyset$. Hence $|U\cap R|\le 2$ and $U\cap L=\emptyset$.
Let $M\subset \PP^n$ be a general hyperplane containing $R$ (so $M=R$ if $n=2$). Since $S$ is a finite set and $M$ is a
general hyperplane containing $S$, we have $M\cap S =R\cap S$. Consider the residual exact sequence (\ref{eqb+2}) of $M$ in $\PP^n$.

\quad (a3.1)  Assume $h^1(\Ii _{S\setminus S\cap M}(d-1))>0$. Since $|S\setminus S\cap M|\le |A|+|B| -d-2 \le 3(d-1)-1$, either there is a line $R_1$ such that
$|R_1\cap (S\setminus S\cap M)| \ge d+1$ or there is a conic $D_1$ such that $|D_1\cap (S\setminus S\cap M)| \ge 2d$. If $R$ and $R_1$ (resp. $R$ and $D_1$) are contained
in a plane, and in particular if $n=2$,  step (a2) (resp. step (a1)) gives a contradiction, because $|S\cap (R\cup R_1)| \ge 2d+3$ (resp. $|S\cap (R\cup D_1)|\ge 3d+2$). Thus we may assume that this is not the case and in particular we may assume $n>2$. Let $N$ be a general hyperplane containing $R_1$ (resp. $D_1$). We use the residual exact sequence (\ref{eqb+4}). Note that $S\cap (M\cup N) = S\cap (R\cup R_1)$ (resp. $S\cap (M\cup N) = S\cap (R\cup \langle D_1\rangle)$.

\quad (a3.1.1) Assume $h^1(\Ii _{S\setminus S\cap (M\cup N)}(d-2))>0$. We exclude the existence of $D_1$, because $|S\cap (R\cup D_1)|\ge 3d+2$ and hence
$|S\setminus S\cap (M\cup N)|\le d-1$. Thus in this case we may assume the existence of $R_1$. Since $|S\cap (R\cup R_1)| \ge 2d+3$, we have $|S\setminus S\cap (M\cup N)|\le
|A|+|B|-2d-3 \le 2(d-2)+1$. By \cite[Lemma 34]{bgi} there is a line $R_2$ such that $|R_2\cap S\setminus S\cap (M\cup N)|\ge d$. Let $M'$ be a general hyperplane containing
$R_2$. Consider the residual exact sequence of $M'\cup M\cup N$. We have $h^1(\Ii _{S\setminus S\cap (M\cup N\cup
M')}(d-3))=0$, because $|S\setminus S\cap (M\cup N\cup M')|\le 2k-d-2-d-1-d \le d-4$. Either \cite[Lemma 5.1]{bb2} or
\cite[Lemmas 2.4 and 2.5]{bbcg1} give
$A\setminus A\cap (M\cup N\cup M') =B\setminus B\cap (M\cup N\cup M')$. Since $M$, $N$ and $M'$ are general, we have $S\cap (M\cup N\cup M') = S\cap (R\cup R_1\cup
R_2) $. Since $U$ is general, no $3$ of the points of $U$ are collinear. Thus $|U\cap (R\cup R_1\cup R_2)|\le 6$. Hence
$|A\setminus A\cap (M\cup N\cup M')|\ge k-d-8+b$. Since
$A\setminus A\cap (M\cup N\cup M') =B\setminus B\cap (M\cup N\cup M')$, we get $|S\cap (M\cup N\cup M')|\le 2k-2k+2d+16-2b$. Hence $2d+16-2b \ge 3d+3$, a contradiction.

\quad (a3.1.2) Assume $h^1(\Ii _{S\setminus S\cap (M\cup N)}(d-2))=0$. Either \cite[Lemma 5.1]{bb2} or \cite[Lemmas 2.4 and
2.5]{bbcg1} give
$A\setminus A\cap (M\cup N) =B\setminus B\cap (M\cup N)$. Since $U\cap (M\cup N) =U\cap (R\cup R_1)$, we have $|U\setminus U\cap (M\cup N)|\ge k-d-6+b$. Assume for the moment $L\notin \{R,R_1\}$. We get $|L\cap (M\cup N)|\le 2$. Thus $|A\setminus A\cap (M\cup N)| \ge k+b-8$. Since $A\setminus A\cap (M\cup N) =B\setminus B\cap (M\cup N)$,
we get $|S\cap (M\cup N)| \le 16-b < 2d+3$ (even when instead of $|S|$ we take $2k$). Thus we may assume that either $L = R$ or $L=R'$. In both cases, writing $D:= R\cup R'$ we are in the case
solved in step (a2.1). 

\quad (a3.2) Assume $h^1(\Ii _{S\setminus S\cap M}(d-1))=0$.  Either \cite[Lemma 5.1]{bb2} or \cite[Lemmas 2.4 and
2.5]{bbcg1} give
$A\setminus A\cap M =B\setminus B\cap M$. 

\quad (a3.2.1) Assume $R=L$. We get $U = A\setminus A\cap L = B\setminus B\cap L$. Thus $B = U\cup (B\cap L)$. Since
$\langle
\nu _d(U)\rangle \cap \langle Y\rangle =\emptyset$, $q\in \langle \nu _d(U)\cup Y\rangle$, $q\notin \langle \nu _d(U)\rangle$,
$q\notin \langle \nu _d(Y)\rangle$ (because $U\ne \emptyset$) and $\langle
\nu _d(U)\rangle \cap \langle Y\rangle =\emptyset$, there are uniquely determined $q_1\in \langle \nu _d(U)\rangle$
and $q_2\in \langle Y\rangle$ such that $q\in \langle \{q_1,q_2\}\rangle$. The uniqueness of $q_2$ gives $q_2=q'$. Since
$\langle
\nu _d(U)\rangle \cap \langle Y\rangle =\emptyset$ and $q\in \langle \nu _d(A)\rangle \cap \langle \nu _d(B)\rangle$,
we get $q'\in \langle \nu _d(B\cap L)\rangle$. Thus $|B\cap L|\ge r_Y(q') =|A\cap L|$. Since $|B|\le |A|$ and $A\setminus
A\cap L = B\setminus B\cap L$, we get $|B|=|A|$ and $B = U\cup F$ with $F\cap U=\emptyset$ and $F\in \Ss (Y,q')$. Thus the
theorem is true in this case.

\quad (a3.2.2) Assume $R\ne L$. Since $|L\cap R|\le 1$, we get $|E\cap R|\le 1$. Since $|U\cap R|\le 2$, we get $|A\cap R|\le
3$ and hence $|B\cap R|\ge d-1> |A\cap R|$. Since $A\setminus A\cap R =B\setminus B\cap R$, we get $|B|>|A|$, a contradiction.

\quad (b) By step (a) we always have $r_X(q)=r_Y(q)$ and hence $\Ss (X,q)\supseteq \{U\cup E\}_{E\in \Ss (L,q')}$, in which we
write
$\Ss (L,q')$ for the set of all $E\subset L$ such that $\nu _d(E)\in \Ss (Y,q')$. Note that $U = \cap _{E\in \Ss (L,q')}
\{U\cup E\}_{E\in
\Ss (L,q')}$. Thus $U$ is uniquely determined by $\Ss (Y,q')$. By step (a) the theorem is true if $k\le 2d-3$. Thus we may assume $k=2d-2$. We fix $E\in \Ss (L,q')$ and $U\in \PP^n$ such that $U\cap L =\emptyset$, $U$ has general Hilbert function and
set $A:= U\cup E\in \Ss (X,q)$. We only need to prove that $L$ is the only line $J$ such that there are $q''\in \langle \nu
_d(J)\rangle$ and infinitely many $S_\alpha \in \Ss (X,q)$,
$\alpha
\in
\Gamma$,
$\Gamma$ an integral quasi-projective curve, $S_\alpha = U'\cup E_\alpha$ for all $\alpha \in \Gamma$, $E_\alpha \subset J$,
$|E_\alpha| =d+2-b$,
$\langle \nu _d(U')\rangle \cap \langle \nu _d(J)\rangle =\emptyset$, $E_\alpha \in \Ss (\nu _d(J),q'') = \Ss (X,q'')$ (by
Autarky), $\cap _\alpha \langle \nu _d(E_\alpha)\rangle = \{q''\}$ and $J\ne L$. Since $\cap _\alpha \langle \nu
_d(E_\alpha)\rangle = \{q'\}\notin X$, we have $\cap _\alpha E_\alpha =\emptyset$. Thus there is $E_\alpha$ such that
$E_\alpha \cap A =\emptyset$ and we call $F$ this set $E_\alpha$. Set $B:= U'\cup F$ and $S:= A\cup B$. Thus $h^1(\Ii _S(d))
>0$. Since $J \ne L$, we have $|J\cap L|\le 1$ and equality holds if $n=2$. Set $D:= J\cup L$. We have $|U| =|U'| = d-4+b$. If
$b=2$ we assume until step (b5) that $|U'\cap L| \le 1$ and $|U\cap J|\le 1$.

\quad (b1) Assume $n=2$. Consider the residual exact sequence of (\ref{eqb+3}). First assume $h^1(\Ii _{S\setminus S\cap D}(d-2)) =0$. By either \cite[Lemma 5.1]{bb2} or \cite[Lemmas 2.4 and 2.5]{bbcg1} we have $A\setminus A\cap D = B\setminus B\cap D$. At most two
points of $U'$ are contained in $L$ and the same number of points are contained in $J$. Since $h^1(\Ii _{U\cup E}(t)) = 0$ for
all $t\ge d+1-b$, we get $h^1(\Ii _{S\setminus S\cap J}(d-1)) =0$ and so $A\setminus A\cap J = U'$, contradicting the
inequality $|E\cap J| \le |L\cap J| =1$. Thus $h^1(\Ii _{S\setminus S\cap D}(d-2)) >0$. Since $S\setminus S\cap D\subseteq
U\cup U'$, $|U\cup U'| \le 2d-8+2b$, at most $2$ (resp. $5$, resp. $9$) points of $U$ or $U'$ are contained in a line (resp. a
conic, resp. a cubic), we get a contradiction.

\quad (b2) Assume $n>2$ and $J\cap  L \ne \emptyset$. Let $N:= \langle L\cup J\rangle$ be the plane spanned by the conic $L\cup J$. Let $M\subset \PP^n$ be a general hyperplane containing $N$. Since $S$ is a finite set and $M$ is general, we have $S\cap M = S\cap N$. Note that $|U\cap N|\le 3$ and $|U'\cap N|\le 3$. We use the residual exact sequence (\ref{eqb+3}) of $M$ in $\PP^n$. First assume $h^1(\Ii _{S\setminus S\cap  M}(d-1)) =0$. We get $h^1(M,\Ii _{S\cap M,M}(d)) >0$. Since $S\cap M = S\cap N$, we have
$h^1(N,\Ii _{S\cap N}(d)) >0$. Since $|U\cap N|\le 3$ and $|U'\cap N|\le 3$, as in step (b1) the residual exact sequence of $J$ in $N$ gives a contradiction.
Now assume $h^1(\Ii _{S\setminus S\cap M}(d-1)) >0$. Since $S\setminus S\cap M \subseteq U\cup U'$ and $|U\cup U'| \le 2d-8+2b$, we get a contradiction by Lemma \ref{a4.5}.

\quad (b3) Assume $J\cap L =\emptyset$ and $n=3$. Let $Q\subset \PP^3$ be a general quadric surface containing $I\cup J$. Since $\Ii _{J\cup L}(2)$ is globally generated,
we have $Q\cap S = S\cap (I\cup J)$. Our assumption on $|U\cap J|$ and on $|U'\cap L|$ gives $h^1(Q,\Ii _{S\cap Q}(d)) =0$.
The residual exact sequence of $Q$ in $\PP^3$ gives
$h^1(\Ii _{S\setminus S\cap Q}(d-2)) >0$. Since $S\setminus S\cap Q\subseteq U\cup U'$ and both $U$ and $U'$ are general (each
of them have no $3$ points collinear, no $6$ points in a conic), Lemma \ref{a4.5} gives a contradiction.

\quad (b4) Assume $J\cap L =\emptyset$ and $n>3$. We use induction on $n$, assuming the result in $\PP^{n-1}$. We have $\dim \langle J\cup L\rangle =3$. Take a general hyperplane
$M\subset  \PP^n$ containing $\langle J\cup J\rangle$ and mimic the proof of step (b2).

\quad (b5) Now assume $b=2$ and that either $|U\cap J| \ge 2$ or $|U'\cap L| \ge 2$. Since $U$ and $U'$ have general Hilbert function, each of the occurring inequalities is an equality. Instead of $E$ we take the unique degree $2$ scheme $Z_1\subset L$ evincing the cactus rank of $q'$ and $Z_2$ for the corresponding one for the point $q''$
such that $\{q''\} := \langle \nu _d(J)\rangle \cap \langle \nu _d(U)\cup \{q\}\rangle$. Set $Z:= Z_1\cup Z_2\cup U\cup U'$. We have $\deg (Z) \le |U|+|U'| +4 \le 2d+4$. By \cite[Lemma 1]{bb} we have $h^1(\Ii _Z(d)) >0$. Since $\deg (Z)<3d$, either there is a line $R$ such that $\deg (R\cap Z)\ge d+2$ or there is a conic $D'$ such that $\deg (D'\cap Z) \ge 2d+2$ (if  $n=2$ this is true by \cite[Corollaire 2]{ep}; if $n>2$ we take a general hyperplane containing $L$ and get a contradiction). Both
possibilities are excluded by the assumptions on $U$ and $U'$ (at most $2$ points of $U'$ in a line and none in $J$, at most
$5$ points in a conic). 
\end{proof}

\section{Irredundantly spanning sets}\label{Sq}

\begin{lemma}\label{q2}
If $r+1-\dim X \le t\le r$, then $\Ss (X,q,t)\ne \emptyset$.
\end{lemma}

\begin{proof}
The case $t=r+1-\dim X$ is an obvious consequence of the proof of \cite[Proposition 5.1]{lt}. Assume $r+2-\dim X \le
t\le r$. Let $Y\subset \PP^r$ be the intersection of $X$ and $(t+\dim X -r-1)$ general quadric hypersurfaces. By Bertini's
theorem
$Y$ is an integral and non-degenerate subvariety of $\PP^r$. Thus for any $q$ we have $\Ss (X,q,t)\supseteq \Ss (Y,q,t)$.
Since $t = r+1-\dim Y$, we get $\Ss (Y,q,t)\ne \emptyset$.
\end{proof}

\begin{remark}\label{q1}
Let $X\subset \PP^d$, $d\ge 4$, be a rational normal curve. Fix $q\in \PP^d$ such that $r_X(q)=2$. Since any subset of $X$
with cardinality at most $d+1$ is linearly independent, the definition of irredundantly spanning set gives $\Ss (X,q,t)
=\emptyset$ for all $t$ such that $3\le t\le d-1$.
\end{remark}

\begin{proof}[Proof of Proposition \ref{q3}:]
Since a finite intersection of non-empty Zariski open subsets of $\PP^r$ is open and non-empty and the interval $\lfloor
(r+2)/2\rfloor \le t \le r $ contains only finitely many integers, it is sufficient to prove the statement for a
fixed $t$. The case $t=r$ is true by Lemma \ref{q1}. The case $r$ even at $t = r/2 +1$ is true by Theorem \ref{a3}. Hence all
cases for $r=4$ are true. Thus we may assume $r\ge 5$ and that the proposition is true for all curves in a lower dimensional
projective space. If $r$ is even we may also assume $t\ne r/2+1$. Fix a general $p\in X$ and call $\ell : \PP^r\setminus \{p\}
\to \PP^{r-1}$ the linear projection from $p$. Let $Y\subset \PP^{r-1}$ be the closure of $\ell (X\setminus \{p\})$ in
$\PP^{r-1}$.
$Y$ is an integral and non-degenerate curve. Since $p$ is general in $X$, it is a smooth point of $X$ and hence $\ell
_{|X\setminus \{p\}}$ exstend to a surjective morphism $\mu : X\to Y$ with $\mu (p)$ associated to the tangent line of $X$ at
$p$. Thus
$Y =
\mu (Y)$. By the trisecant lemma (\cite[Corollary 2.2]{r})  and the generality of $p$ we have $\deg (L\cap X) \le 2$  for every
line
$L\subset
\PP^r$ such that
$p\in L$. Hence
$\ell _{|X\setminus \{p\}}$ is birational onto its image and there a finite set $F\subset X$ containing $p$ such that $\mu
_{|X\setminus F}$ induces an isomorphism between $X\setminus F$ and $Y\setminus \mu (F)$. Fix the integer $t$ such that
$\lfloor (r+2)/2\rfloor \le t \le r $ and write $z:= t-1$. By the inductive assumption and, if $r$ is odd and
$t=\lfloor (r+2)/2\rfloor$, Theorem \ref{a3} applied to the projective space $\PP^{r-1}$ there is a non-empty open subset $\Vv$
of
$\PP^{r-1}$ such that
$W(Y)_{q,z} =\{q\}$ for all $q\in \Vv$. Fix $a\in \Vv$ and finitely many $S_i\in \Ss (Y,a,z)$, $1\le i\le e$, such that $\{a\}
= \cap _{i=1}^{e} \langle S_i\rangle$. Restricting if necessary $\Vv$ we may assume that (for a choice of sufficiently general
$S_1(a),\dots ,S_e(a)$)  we have $S_i(a)\cap \mu (F)=\emptyset$ for all $i$ and all $a$. Hence there is a unique
$A_i(a)\subset X\setminus F$ such that $\mu (A_i(a)) =S_i(a)$. Since $p\in F$, $B_i(a):= A_i(a)\cup \{p\}$ has cardinality $t$,
$1\le i\le e$. Set $\Uu _p:= \ell ^{-1}(\Vv)\subset \PP^r\setminus \{p\}$. For each $a\in \Vv$, set $L_a:= \{p\} \cup \ell
^{-1}(a)$. Each $L_a$ is a line containing $p$, $\Uu _p$ is the union of all $L_a\setminus \{p\}$, $a\in \Vv$, and $L_a =\cap
_{i=1}^{e} \langle B_i(a)\rangle$. Fix $a\in \Vv$ and $b\in L_a\setminus \{p\}$. Note that each $B_i(a)$ irredundantly spans
$b$. Fix another general $o\in X$, $o\ne p$. We get in the similar way a set $\Uu _o$. It is easy to check that $W_{q,t}
=\{q\}$ for all $q\in \Uu _o\cap \Uu _p$. Thus we may take $\Uu =\Uu _p\cap \Uu _o$.\end{proof}

\providecommand{\bysame}{\leavevmode\hbox to3em{\hrulefill}\thinspace}

\end{document}